\documentclass{article}
\usepackage[english]{babel}
\usepackage{amsmath}
\usepackage{amssymb}
\usepackage{amsfonts}
\usepackage{amsthm}
\usepackage{graphicx}

\usepackage{url}

\def\thtext#1{
  \catcode`@=11
  \gdef\@thmcountersep{. #1}
  \catcode`@=12
}

\def\threst{
  \catcode`@=11
  \gdef\@thmcountersep{.}
  \catcode`@=12
}

\theoremstyle{plain}
\newtheorem{thm}{Theorem}[section]
\newtheorem{prop}[thm]{Proposition}
\newtheorem{cor}[thm]{Corollary}

\theoremstyle{definition}

 \catcode`@=11
 \def\.{.\spacefactor\@m}
 \catcode`@=12

\def\R{\mathbb R}

\def\a{\alpha}

\def\e{\varepsilon}

\def\D{\Delta}

\def\l{\lambda}

\def\0{\emptyset}
\def\:{\colon}
\def\<{\langle}
\def\>{\rangle}

\def\rom#1{\emph{#1}}
\def\({\rom(}
\def\){\rom)}

\def\ss{\subset}

\def\x{\times}

\def\diam{\operatorname{diam}}

\def\cD{{\cal D}}

\def\cM{{\cal M}}

\begin{document}
\title{Solution to Generalized Borsuk Problem in Terms of the Gromov--Hausdorff Distances to Simplexes}
\author{A.\,O.~Ivanov and A.\,A.~Tuzhilin}
\date{}
\maketitle

\begin{abstract}
In the present paper the following Generalized Borsuk Problem is studied: Can a given bounded metric space $X$ be partitioned into a given number $m$ (probably an infinite one) of subsets, each of which has a smaller diameter than $X$? We give a complete answer to this question in terms of   the Gromov--Hausdorff distance from $X$ to a simplex of cardinality $m$ and having a diameter less than $X$. Here a simplex is a metric space, all whose non-zero distances are the same.
\end{abstract}

\section*{Introduction}
\markright{\thesection.~Introduction}
The present paper is devoted to well-known problem, whose history goes up to Karol Borsuk: How many parts one needs to partition an arbitrary subset of the Euclidean space into, to obtain pieces of smaller diameters? In 1933 Borsuk made the following famous conjecture: Any bounded non-single-point subset of $\R^n$ can be partitioned into $n+1$ subsets, each of which has smaller diameter than the initial subset. This conjecture was proved by H.~Hadwiger~\cite{Hadw} and~\cite{Hadw2} for convex subsets with smooth boundaries, and then suddenly disproved in general case in 1993 by J.~Kahn, and G.~Kalai, see~\cite{KahnKalai}. The current state is described, for example, in~\cite{Raig}.

On the other hand, Lusternik and Schnirelmann~\cite{LustSch}, and a bit later independently Borsuk~\cite{BorCong} and~\cite{Borsuk}, see details, for example, in~\cite{Zieg}, have shown that the standard sphere and the standard ball in $\R^n$, $n\ge2$, cannot be partitioned into  $m\le n$ subsets having smaller diameters.

In the present paper we consider a generalized Borsuk problem passing to an arbitrary bounded metric space  $X$ and partitions of an arbitrary cardinality $m$ (not necessary a finite one). We give a criterion solving the Borsuk problem in terms of the famous Gromov--Hausdorff distance. It turns out that to verify the Borsuk problem solvability it suffices to calculate the Gromov--Hausdorff distance from the space $X$ to a simplex of the cardinality $m$ having a smaller diameter than $X$. Here simplex is a metric space, all whose non-zero distances are the same.

Notice that paper~\cite{IvaTuzSimpDist} and its generalization~\cite{GrigIvaTuz} are devoted to the calculation of this distance. Relations between the distances of this type and other geometrical problems is also demonstrated in paper~\cite{TuzMST-GH}, where in terms of the distances from a finite metric space $X$ to finite simplexes the edges lengths of a minimal spanning tree constructed on $X$ are calculated.

The work is partly supported by President RF Program supporting leading scientific schools of Russia (Project NSh--6399.2018.1, Agreement~075--02--2018--867), by RFBR, Project~19-01-00775-a, and also by MGU scientific schools support program.

\section{Preliminaries}
\markright{\thesection.~Preliminaries}
Let $X$ be an arbitrary set. By $\#X$ we denote the \emph{cardinality\/} of the set $X$.

Let $X$ be an arbitrary metric space. The distance between any its points $x$ and $y$ we denote by $|xy|$. If $A,B\ss X$ are non-empty subsets of $X$, then put $|AB|=\inf\bigl\{|ab|:a\in A,\,b\in B\bigr\}$. For $A=\{a\}$, we write $|aB|=|Ba|$ instead of $|\{a\}B|=|B\{a\}|$.

For each  point $x\in X$ and a number $r>0$, by $U_r(x)$ we denote the open ball with center $x$ and radius $r$; for any non-empty $A\ss X$ and a number $r>0$ put $U_r(A)=\cup_{a\in A}U_r(a)$.

\subsection{Hausdorff and Gromov--Hausdorff Distances}
For non-empty $A,\,B\ss X$ put
\begin{multline*}
d_H(A,B)=\inf\bigl\{r>0:A\ss U_r(B),\ \text{and}\ B\ss U_r(A)\bigr\}\\ =\max\{\sup_{a\in A}|aB|,\ \sup_{b\in B}|Ab|\}.
\end{multline*} 
This value is called the \emph{Hausdorff distance between $A$ and $B$}. It is well-known, see~\cite{BurBurIva}, \cite{ITlectHGH}, that the Hausdorff distance is a metric on the set of all non-empty bounded closed subsets of $X$.

Let $X$ and $Y$ be metric spaces. A triple $(X',Y',Z)$ consisting of a metric space $Z$ together with its subsets $X'$ and $Y'$ isometric to $X$ and $Y$, respectively, is called a \emph{realization of the pair $(X,Y)$}. The \emph{Gromov--Hausdorff distance $d_{GH}(X,Y)$ between $X$ and $Y$} is the infimum of real numbers $r$ such that there exists a realization  $(X',Y',Z)$ of the pair $(X,Y)$ with $d_H(X',Y')\le r$. It is well-known~\cite{BurBurIva}, \cite{ITlectHGH}, that $d_{GH}$ is a metric on the set $\cM$ of all compact metric spaces considered up to an isometry.

For an arbitrary metric space $X$ and a number $\l>0$ by $\l X$ we denote the metric space that is obtained from $X$ by the multiplication of all its distances by $\l$. The metric space of cardinality $m$, such that all its non-zero distances are equal to $1$ is called the  \emph{unit simplex\/} and is denoted by $\D_m$. Each space $\l\D_m$, $\l>0$, is referred as  \emph{simplex}. Notice that $\D_1$ is the single-point metric space.

For an arbitrary metric space $X$ by $\diam X$ we denote its \emph{diameter\/} defined in the standard way:
$$
\diam X=\sup\bigl\{|xy|:x,y\in X\bigr\}.
$$
Notice that the space $X$ is bounded, if and only if $\diam X<\infty$.

\begin{prop}[\cite{BurBurIva}, \cite{ITlectHGH}]\label{prop:GH_simple}
For an arbitrary metric space $X$ we have
$$
2d_{GH}(\D_1,X)=\diam X.
$$
\end{prop}

We need the following two Theorems from~\cite{GrigIvaTuz}, which the main results of the present paper is based on.

\begin{thm}[{\cite[Theorem~2.1]{GrigIvaTuz}}]\label{thm:dist-n-simplex-bigger-dim}
Let $X$ be an arbitrary bounded metric space, and $\#X<\#\l\D$, then
$$
2d_{GH}(\l\D,X)=\max\{\l,\diam X-\l\}.
$$
\end{thm}

Let $X$ be an arbitrary set and $m$ a cardinal number that does not exceed $\#X$. By  $\cD_m(X)$ we denote the family of all possible partitions of the set $X$ into $m$ non-empty subsets.

Now let $X$ be a metric space. Then for each $D=\{X_i\}_{i\in I}\in\cD_m(X)$ put
$$
\diam D=\sup_{i\in I}\diam X_i.
$$
Further, for any non-empty $A,B\ss X$ put $|AB|=\inf\bigl\{|ab|:(a,b)\in A\x B\bigr\}$,
and for each $D=\{X_i\}_{i\in I}\in\cD_m(X)$ put
$$
\a(D)=\inf\bigl\{|X_iX_j|:i\ne j\bigr\}.
$$

\begin{thm}[{\cite[Proposition~2.4 and Corollary~2.3]{GrigIvaTuz}}]\label{thm:GH-dist-alpha-beta}
Let $X$ be an arbitrary bounded metric space and $m$ a cardinal number such that $m\le\#X$. Then
$$
2d_{GH}(\l\D_m,X)=\inf_{D\in\cD_m(X)}\max\{\diam D,\,\l-\a(D),\,\diam X-\l\}.
$$
\end{thm}

\section{Generalized Borsuk Problem}
\markright{\thesection.~Generalized Borsuk Problem}

Classical Borsuk Problem deals with partitions of subsets of Euclidean space into parts having smaller diameters, see details in Introduction. Now, generalize Borsuk Problem to arbitrary bounded metric spaces and partitions of arbitrary cardinality. Let $X$ be a bounded metric space, $m$ a cardinal number such that $m\le\#X$, and $D=\{X_i\}_{i\in I}\in\cD_m(X)$. We say that  $D$ is a partition into subsets having \emph{strictly smaller diameters}, if there exists $\e>0$ such that $\diam X_i\le\diam X-\e$ for all $i\in I$.

By \emph{Generalized Borsuk Problem\/} we call the following one: Can a given bounded metric space $X$ be partitioned into a given, probably infinite, number of subsets, each of which has a strictly smaller diameter than $X$?

We give the folloiwng solution to the Generalized Borsuk Problem in terms of the Gromov--Hausdorff distance.

\begin{thm}\label{thm:main}
Let $X$ be an arbitrary bounded metric space and $m$ a cardinal number such that $m\le\#X$. Choose an arbitrary number $0<\l<\diam X$, then $X$ can be partitioned into $m$ subsets having  strictly smaller diameters, if and only if  $2d_{GH}(\l\D_m,X)<\diam X$.
\end{thm}

\begin{proof}
For the $\l$ chosen, due to Theorem~\ref{thm:GH-dist-alpha-beta}, we have $2d_{GH}(\l\D_m,X)\le\diam X$, and the equality holds, if and only if for each $D\in\cD_m(X)$ the equality $\diam D=\diam X$ is valid. The latter means that there is no partition of the space $X$ into $m$ parts having strictly smaller diameters.
\end{proof}

Next Corollary gives a geometrical description of the spaces which Generalized Borsuk Problem has no solution for.

\begin{cor}
Let $d>0$ be a real number, and $m\le n$ cardinal numbers. By $\cM_n$ we denote the set of isometry classes of bounded metric spaces of cardinality at most $n$, endowed with the Gromov--Hausdorff distance. Choose an arbitrary $0<\l<d$. Then the intersection
$$
S_{d/2}(\D_1)\cap S_{d/2}(\l\D_m)
$$
of the spheres, considered as the spheres in $\cM_n$, does not contain spaces of cardinality less than $m$, and consists exactly of all the metric spaces from $\cM_n$, whose diameters are equal to $d$ and that cannot be partitioned into $m$ subsets of strictly smaller diameters.
\end{cor}

\begin{proof}
Let $X$ belong to the intersection of the spheres, then $\diam X=d$ in accordance with Proposition~\ref{prop:GH_simple}. If $m>\#X$, then, due to Theorem~\ref{thm:dist-n-simplex-bigger-dim}, we have
$$
2d_{GH}(\l\D,X)=\max\{\l,\diam X-\l\}<d,
$$
therefore $X\not\in S_{d/2}(\l\D_m)$, that proves the first statement of Corollary.

Now let $m\le\#X$. Since $\diam X=d$ and $2d_{GH}(\l\D_m,X)=d$, then, due to Theorem~\ref{thm:main}, the space $X$ cannot be partitioned into $m$ subsets of strictly smaller diameters.

Conversely, each $X$ of the diameter $d$, such that $m\le\#X$ and that cannot be partitioned into $m$ subsets of strictly smaller diameter lies in the intersection of the spheres by Theorem~\ref{thm:main}.
\end{proof}

\markright{References}

\end{document}